\documentclass[10pt]{article}
\usepackage[T1]{fontenc}
\usepackage{lmodern}
\usepackage{multicol}
\usepackage{dsfont}
\usepackage{graphicx}
\usepackage{amsthm}
\usepackage{amssymb,amsmath}
\numberwithin{equation}{section}
\theoremstyle{definition}
\newtheorem{theorem}{Theorem}
\newtheorem{mydef}[theorem]{Definition}
\newtheorem{corollary}[theorem]{Corollary}
\newtheorem{example}[theorem]{Example}
\newtheorem{lemma}[theorem]{Lemma}
\newtheorem{proposition}[theorem]{Proposition}
\usepackage{color}
\newcommand{\Ito}{It\^{o}}
\newcommand{\Holder}{H\"{o}lder}
\newcommand{\thesistitle}{Fractional Brownian Motion and the Fractional Stochastic Calculus}
\newcommand{\thesisauthor}{Ben McGonegal}
\newcommand{\thesisadvisor}{Advisor: Professor Tom LaGatta}
\newcommand{\graddate}{May/2013}
\usepackage{setspace}

\begin{document}

\pagenumbering{arabic}

\thispagestyle{empty}
 
\begin{center}

  {\large\textbf{\thesistitle}}
  \vspace{.2in}

 \thesisauthor
 
 \vspace{.6in}
 \begin{doublespace}
 
 New York University, Courant Institute of Mathematical Sciences\\
 251 Mercer Street\\
 New York, NY  10012\\
  \graddate
  \end{doublespace}
  
  \vspace{.6in}

  A thesis submitted in partial fulfillment\\
  of the requirements for the degree of\\
  master's of science\\
  Department of Mathematics\\
  New York University\\
  \graddate

\end{center}
\vfill

\noindent\makebox[\textwidth]{\hfill\makebox[2.5in]{\hrulefill}}\\
\makebox[\textwidth]{\hfill\makebox[2.5in]{\hfill\thesisadvisor\hfill}}
\newpage

\tableofcontents

\newpage

\begin{doublespace}

\section{Introduction}\label{intro}

	In recent years, interest in fractional Brownian motion (fBm) has grown mostly due to applications in such fields as hydrology, economics, telecommunications and finance. Embrechts and Maejima in [13] and Dai and Heyde in [28] both explain that due to this popularity, demand for a stochastic calculus with respect to fractional Brownian motion has increased. Fractional Brownian motion, also known as a fractal Brownian motion, is comparable to a continuous fractal random walk. Though, unlike regular Brownian motion, fBm has dependent increments, which means that the current "step" of a fBm is dependent on previous "steps." This dependence is measured on a scale from zero to one and this measure is called the Hurst index, $H \in (0,1),$ named after hydrologist Harold Edwin Hurst for his work in the field of hydrology.  
	
	Hurst studied the yearly variance in levels of the Nile river and applied this to the so called $R/S$ statistic, where $R$ is the range of partial sums of the data and $S$ is the sample standard deviation.  The $R/S$ statistic should grow like $n^{1/2}$ under normal assumptions of independent and identically distributed observations and finite variance, where $n$ is the sample size. Interestingly enough, the Nile data indicated growth of $n^H$, where $H \in (1/2, 1)$.  Random walk typically yields a growth of $n^{1/2}$, and the scaling limit of random walk in dimension one is Brownian motion. Hence, it must be the case that the growth $n^H$, with $H \in (1/2, 1)$ corresponds to something else [31].  Mandelbrot noticed that while Brownian motion has standard deviation $t^{1/2}$, fractional Brownian motion has a standard deviation of $t^H$, where $0 < H < 1$, and thus fBm might be a more appropriate fit for this behavior [32]. 

	The Hurst index describes the raggedness of the path of the fBm it is associated with, where a value of $1/2$ corresponds to non-correlating increments.  A value greater than $1/2$ correspond to positive correlation. Heuristically speaking, if a process corresponding to a value of $H$ that falls within this range is going up during an interval, then it will likely keep going up in the next interval.  On the other hand, values less than $1/2$ correspond to negative correlation.  Furthermore, if a process that has positively correlated increments has upward growth in an interval, then it will likely move down in the next interval.  A useful application of fBm for values of $H \in (1/2, 1)$ is in describing the behavior to prices of assets and volatilities in stock markets [23].  Fractional Brownian motion, written $B^H (t),$ is a generalization of Brownian motion, which is a fBm with Hurst index $H = 1/2.$  It turns out fractional Brownian motions divide into these three very different cases, corresponding to the interval $H$ is associated to. For this reason, we present these cases separately. First, the classical case of Brownian motion when $H = 1/2$ and then the case when $H \neq 1/2$. 
	
		In 1940, it was Andrei Kolmogorov, while studying spiral curves in Hilbert space, who first introduced fractional Brownian motion.  However, it wasn't until Mandelbrot recognized fBm's significance that he, together with Van Ness, derived many of its important properties in their famous paper [29] in 1968. It was in that paper that fractional Brownian motion was given its name, which comes from its representation as a fractional stochastic integral with respect to Brownian motion. An integral is called stochastic when either the integrator, integrand, or both are stochastic processes, thus making itself a random process as well. A fractional integral is one where we take an $\alpha$-tuple iterated integral, where $\alpha$ need not be an integer. As we shall show later, a fBm is nothing more than a ($H - 1/2$)-tuple iterated integral of a  regular Brownian motion.  In fact, in [29] Mandelbrot and Van Ness describe $B^H(t)$, for values of $H \neq 1/2$, as being the "fractional derivative or integral of $B^{\frac{1}{2}}(t)$."  In Section \ref{fraccalc}, we present the fractional calculus and discuss some of its definitions. We do this because of its strong relationship and importance to fractional Brownian motion.  It is interesting to note that a fBm, $B^H(t),$ can be written (as we will show) in terms of the fractional calculus operators defined, and hence we can take full advantage of the properties inherent within the fractional calculus in order to show some properties that fBm exhibits.
		
	We begin in Section \ref{fraccalc} by introducing the fractional calculus, from the Riemann-Liouville perspective. In Section \ref{BM}, we introduce Brownian motion and its properties, which is the framework for deriving the {\Ito} integral. In Section \ref{itogral} we finally introduce the {\Ito} calculus and discuss the derivation of the {\Ito} integral. Section \ref{itormula} continues the discussion about the {\Ito} calculus by introducing the {\Ito} formula, which is the analogue to the chain rule in classical calculus. In Section \ref{FBM} we present our formal definition of fBm and derive some of its properties that give motivation for the development of a stochastic calculus with respect to fBm. Finally, in Section \ref{fbmstochcalc} we define and characterize a stochastic integral with respect to fBm from a pathwise perspective.

\pagebreak

\section{Fractional Calculus}\label{fraccalc}

Fractional calculus is a branch of mathematical analysis that unifies the integration operator and differentiation operator of classical calculus as one operator, the differintegral. The differintegral is a single operator depending on a real valued parameter $\alpha,$ where positive values of $\alpha$ correspond to differentiation and negative values of $\alpha$ correspond to integration.  Fractional calculus is an extension, or generalization, of the well known classical calculus. It gets its name from the idea that instead of taking integer order derivatives and integrals, what happens when we take fractional (or any real number) orders of differintegrals. As we will see, when we take integer ordered derivatives and integrals of the well defined fractional calculus, we get just that, the classical calculus first derived by Leibniz and Newton. As mentioned above, our interest in fractional calculus within the framework of this paper is due to the fact that fBm can be represented by a fractional stochastic integral.  We first begin by introducing the fractional integral, derived from a $n$-tuple [integer ordered] iterated integral, and express it as a single integral dependent on the parameter $n.$

Let $f$ be some function on the interval $[a,b]$, then a multiple integral of $f$ can be expressed by Cauchy's formula for repeated integration:

$$\int^{x_n}_a \int^{x_{n-1}}_a \cdots \int^{x_1}_a f(u) du dx_1 \cdots dx_{n-1} $$
\begin{equation}
= \frac{1} {(n-1)!} \int^{x_n}_a f(u) (x_n - u)^{n-1} du,
\label{cauchy}
\end{equation}
where $x_n \in [a,b]$ and $n \in \mathbb{Z}^{+}.$ This shows nicely the correspondence between the number of integrals we are integrating over the integrand with on the left side of \eqref{cauchy} with the number in the denominator and the exponent in the integrand on the right side of \eqref{cauchy}. Expressing the repeated integral as in \eqref{cauchy} gives us the framework to be able to integrate (or differentiate) a function a fraction amount of times.  For example, if we wanted to integrate a function one and a half times, we would simply write down $1.5$ everywhere we see $n$, but this causes a problem. Notice we are computing a factorial dependent on $n$, which is only defined for positive integers, yet we are interested in taking factorials of any real number (as in the example we just mentioned of  3/2). We will remedy this problem after we give the proof of \eqref{cauchy}, of which we will walk through because it is insightful.  We will also revisit some of the ideas in Sections \ref{FBM} and \ref{fbmstochcalc}.

\begin{proof}
A proof is given by induction.  Consider the base case, where $n=2$. Then,

$$\int^{x_2}_a \int^{x_1}_a f(u) du dx_1 = \int^{x_2}_a f(u) \int^{x_2}_u dx_1 du = \int^{x_2}_a f(u) (x_2 - u) du $$
by changing the order of integration. Now suppose this is true for $n-1$, then by changing the order of integration\\

$\displaystyle \int^{x_n}_a \int^{x_{n-1}}_a \cdots \int^{x_1}_a f(u) du dx_1 \cdots dx_{n-1}$

\begin{align*}
 & = \frac{1} {(n-2)!} \int^{x_n}_a \int^{x_{n-1}}_a f(u) (x_{n-1} - u)^{n-2} du dx_{n-1}\\
 & = \frac{1} {(n-2)!} \int^{x_n}_a  f(u) \int^{x_{n-1}}_u (x_{n-1} - u)^{n-2} dx_{n-1} du\\
 & = \frac{1} {(n-1)!} \int^{x_n}_a f(u) (x_n - u)^{n-1} du.
 \end{align*}
 Hence by induction, \eqref{cauchy} is true.
 \end{proof}

By observing that $(n-1)! = \Gamma(n)$, where $\Gamma$ is the gamma function, and replacing $n$ in \eqref{cauchy} with $\alpha \in \mathbb{R}^{+},$ we can define the Riemann-Liouville fractional integrals (a detailed derivation can be found in [3]). This structure allows us take an $\alpha$-tuple iterated integral of a function $f,$ for any real valued $\alpha,$ hence the name {\it{fractional}} integral (though this is a bit misleading, as we will define this integral for not only real numbers that can be expressed as fractions, but those that can not as well!).

\begin{mydef}
Let $f \in L^1[a,b]$, $\alpha \in \mathbb{R}^{+}$ and $t \in (a,b)$. The fractional integrals of order $\alpha$ on the intergal $[a,b]$ are
\begin{equation}
 \label{Ia+}
 (I^\alpha_{a+} f)(t) = \frac{1} {\Gamma(\alpha)} \int^{b}_a f(u) (t - u)^{\alpha -1}_+ du
\end{equation}
and
\begin{equation}
 \label{Ib-}
 (I^\alpha_{b-} f)(t) = \frac{1} {\Gamma(\alpha)} \int^{b}_a f(u) (u - t)^{\alpha -1}_+ du.
 \end{equation}
\end{mydef}

\vspace{10 mm}
Notice that in \eqref{Ia+}, $(t - u)^{\alpha -1}_+ = 0$ when $u \geq t$, thus $I^\alpha_{a+}$ can be written as 
\begin{equation}
\label{Ia}
\frac{1} {\Gamma(\alpha)} \int^{t}_a f(u) (t - u)^{\alpha -1} du.
\end{equation}

Similarly, \eqref{Ib-} can be written as

\begin{equation}
\label{Ib}
\frac{1} {\Gamma(\alpha)} \int^{b}_t f(u) (u - t)^{\alpha -1} du.
\end{equation}	

For this reason $I^\alpha_{a+}$ is called left sided, because the interval of integration, $[a, t]$, of \eqref{Ia} is over the left side of the interval $[a, b]$. Similarly, $I^\alpha_{b-}$ is called right sided.

When deriving fractional integrals, we restricted the order of itegration, $\alpha \in \mathbb{R}$ to be strictly positive. Now, to obtain fractional derivatives, we consider the order $\alpha$ to be negative. However, \eqref{Ia} and \eqref{Ib} diverge if we replace $\alpha \in \mathbb{R}^{+} $ with $\alpha \in \mathbb{R}^{-}$. So, if we restrict $0 < \alpha < 1$, we can define the Riemann-Liouville fractional derivatives of order $\alpha$.

\begin{mydef}
Let $f \in L^1[a,b]$, $0 < \alpha < 1$ and $t \in (a,b)$. The fractional derivatives of order $\alpha$ on the interval $[a,b]$ are
\begin{equation}
 (D^\alpha_{a+} f)(t) = \frac{1} {\Gamma(1-\alpha)} \frac{d}{du}\int^{b}_a f(u) (u - t)^{-\alpha}_+ du
\end{equation}
and
\begin{equation}
 (D^\alpha_{b-} f)(t) = \frac{1} {\Gamma(1-\alpha)}  \frac{d}{du}\int^{b}_a f(u) (t - u)^{-\alpha}_+ du.
 \end{equation}
\end{mydef}

\vspace{10 mm}

Furthermore, this case admits what is known as the Weyl representation of the fractional derivatives:

\begin{equation}
 (D^\alpha_{a+} f)(t) = \frac{1} {\Gamma(1-\alpha)} \Bigg[ \frac{f(t)}{(t-a)^{\alpha}} + \alpha \int^{t}_a \frac{f(t) - f(u)}{(t-u)^{\alpha -1}} du \Bigg]
\end{equation}
and
\begin{equation}
(D^\alpha_{b-} f)(t) = \frac{1} {\Gamma(1-\alpha)} \Bigg[ \frac{f(t)}{(b-t)^{\alpha}} + \alpha \int^{b}_t \frac{f(t) - f(u)}{(u-t)^{\alpha -1}} du \Bigg].
 \end{equation}

\vspace{5mm}

Notice that $D^\alpha_{a+} = I^{-\alpha}_{a+}$ and $D^\alpha_{b-} = I^{-\alpha}_{b-}$ (which are easily derived from \eqref{Ia} and \eqref{Ib} in [6]). Furthermore, the fractional derivatives $D^\alpha_{a+}$ and $D^\alpha_{b-}$ are called left sided and right sided, respectively.

\begin{mydef}
Let $\alpha \in (0,1)$ and $f$ a function over $\mathbb{R}.$ Then 
$$(I^\alpha_{-} f)(t) = \frac{1} {\Gamma(\alpha)} \int^{\infty}_{- \infty} f(u) (t - u)^{\alpha -1}_+ du$$
and
$$(I^\alpha_{+} f)(t) = \frac{1} {\Gamma(\alpha)} \int^{\infty}_{- \infty} f(u) (t - u)^{\alpha -1}_- du$$
are called the left-sided and right-sided fractional integrals of order $\alpha$ on the whole real line
\end{mydef}

\begin{example}
Let $f(t) = (t-a)^{\alpha - 1}$, $t \in (a, b)$ and $0 < \alpha < 1.$ Then
\begin{align*}
(D^\alpha_{a+} f)(u) 
& = \frac{1} {\Gamma(1-\alpha)} \frac{d}{du}\int^{u}_a (t-a)^{\alpha - 1} (u - t)^{-\alpha} dt\\
& = \Gamma(\alpha) \frac{d}{du} 1 = 0.
\end{align*}
\end{example}

\subsection{Properties of Fractional Integrals and Derivatives}\label{fraccalcprop}

We now present some properties of interest of the well defined fractional calculus.

\begin{enumerate}
	\item Reflection: Let $R$ be defined by $(Rf)(t) = f(a + b - t)$ for $t \in [a,b]$ and $\alpha \in 	\mathbb{R}^{+}$, then
	$$RI^\alpha_{a+} = I^\alpha_{b-} R$$
	  and  
	  $$RI^\alpha_{b-} = I^\alpha_{a+} R.$$
	R is known as the reflection operator.
	\item Composition: Let $f \in L^1[a,b]$ and $\alpha, \beta >0$, then
	$$I^\alpha_{a+}I^\beta_{a+} f = I^{\alpha + \beta}_{a+}, I^\alpha_{b-}I^\beta_{b-} f = I^{\alpha + \beta}_{b-}, $$
	and
	$$D^\alpha_{a+}D^\beta_{a+} f = D^{\alpha + \beta}_{a+}, D^\alpha_{b-}D^\beta_{b-} f = D^{\alpha + \beta}_{b-}.$$
	\item Integration by parts: Let $f \in L^p[a,b], g \in L^q[a,b],$  with either $\frac{1}{p} + \frac{1}{q} \leq 1 + \alpha,$ where $p >$1 and $q > 1$ or $\alpha \geq 1, p=1, q=1$, then
	
	\begin{equation}
	\label{fracibp}
	 \int^{b}_a f(u) (I^\alpha_{b-}g)(u) du =  \int^{b}_a (I^\alpha_{a+}f)(u) g(u)du.
	\end{equation}
	
	\item Identity: Let $0 < \alpha < 1$, then for any $f \in L^p[a,b],$
	\begin{equation*}
	D^\alpha_{a+} I^\alpha_{a+} f = f,
	\end{equation*}
	and for any $g$ such that $g = I^\alpha_{a+} f,$
	\begin{equation}
	I^\alpha_{a+} D^\alpha_{a+} f = f.
	\end{equation}
\end{enumerate}

The next definition introduces the ${\it{fractal}}$ ${\it{integral.}},$ which we will revisit in Section \ref{fbmstochcalc}.

\begin{mydef}\label{fractegral}
Let $f_{a+} \in I^{\alpha}_{a+}(L^p(a,b)),$ $g_{b-} \in I^{1 - \alpha}_{b-}(L^q(a,b)),$ $1/p + 1/q \leq 1,$ $0 \leq \alpha \leq 1$ and $\alpha p < 1,$ where 
$$f_{a+}  = I_{(a,b)}[f(x) - f(a+)],$$
$$f_{b-}  = I_{(a,b)}[f(x) - f(b-)],$$
provided that $f(a+) = \displaystyle \lim_{x \rightarrow a+} f(x)$ and $f(b-) = \displaystyle \lim_{x \rightarrow b-} f(x)$ exist. Then
\begin{enumerate}
	\item If $g(a+)$ exists and $f \in C^{\alpha - 1/p},$ we define the fractal integral 
	$$\int^b_a f dg = \int^b_a D^{\alpha}_{a+} f_{a+}(x) D^{1- \alpha}_{b-} g_{b-}(x) dx + f(a+)[g(b-) - g(a+)].$$
	\item If $g \in C^{1 - \alpha - 1/q},$ we define the fractal integral 
	$$\int^b_a f dg = \int^b_a D^{\alpha}_{a+} f_{a+}(x) D^{1- \alpha}_{b-} g_{b-}(x) dx.$$
\end{enumerate}
\end{mydef}

	It is important to remark that the fractional integrals are well defined for functions $f \in L^p [a,b],$ $p \geq 1,$ and forms a group of operators as well. Clearly, for any values $\alpha, \beta \in \mathbb{R},$ we get that $I^0_{\cdot} f = f$ and $(I^{\alpha}_{\cdot} \circ I^{\beta}_{\cdot}) f = I^{\alpha + \beta}_{\cdot} f$ (assuming of course the limits of integration are equivalent). Moreover, it is clear to see that as the definition of the fractional derivative implies,  when $\alpha = 1,$ we get the case where $D^1_{\cdot} f = \frac{d}{dx} f,$ hence $D^{\alpha}_{\cdot} f = \frac{d}{dx} I^{1 - \alpha}_{\cdot} f.$

\section{Brownian Motion}\label{BM}

Now that we have introduced fractional calculus, our next aim is to discuss Brownian motion, denoted as $B(t),$ and the {\Ito} integral. Brownian motion gets its name from the Scottish botanist Robert Brown, who noticed while looking through a microscope in 1827 that tiny particles seemed to move randomly through water. It wasn't until Albert Einstein published a paper in 1905 that this random movement of particles suspended in a fluid was explained. Whereas Brownian motion is the actual physical motion of these particles, the Wiener process is the mathematical interpretation of this process. Thus, the Wiener process is synonymous with the standard Brownian motion, the case when $B(0) = 0.$

	As we are building a discussion about fractional Brownian motion and how it is related to fractional calculus, recall that Brownian motion is merely the special case of fBm corresponding to $H = 1/2.$ This classical form of fBm is the only case that exhibits independent increments, hence, when $H = 1/2,$ the correlation of two distinct increments is zero. This separates the generalized fractional Brownian motion into two more categories, the case when $H < 1/2$ and the case when $H > 1/2.$ In later sections, we will discuss in detail the differences between these two cases, but in this section, we focus on the case when $H = 1/2.$ We begin with a formal defintion.

\begin{mydef}
A Brownian motion starting at $a \in \mathbb{R}$ is a real valued stochastic process $\{B(t):t \geq 0\}$ with the following properties:
\begin{enumerate}
	\item $B(0) = a$
	\item $B(t)$ has independent increments
	\item for all $t \geq 0$ and $h > 0,$ $B(t+h) - B(t)$ is normally distributed with expectation zero and variance $h$
	\item $B(t)$ has continuous trajectories almost surely, i.e. $B(t)$ is a continuous function.
\end{enumerate}
Note that if $a = 0$, $B(t)$ is called a standard Brownian motion.
\end{mydef}

Covariance of Brownian motion: $E[B(t)B(s)] = min(s,t).$ The proof follows from property $2$ and $3$ in the definition above. Without loss of generality, let $0 \leq s \leq t,$ then
\begin{align*}
E[B(s)B(t)]  
&= E[B(s)(B(t) - B(s) + B(s))] \\
&=  E[B(s)(B(t) - B(s))] + E[(B(s))^2].\\
\end{align*}
But, by independence of increments and the fact that $E[B(t)] = 0$ for all $t \geq 0$ we have
$$E[B(s)(B(t) - B(s))] =  E[B(s)]E[B(t) - B(s)] = 0.$$
Moreover, since the variance of $B(t)$ is zero for all $t \geq 0$, $E[B(s)^2] = s$, and hence
\begin{align}
E[B(t)B(s)] = min(s,t).
\end{align}

The following two lemmas introduce the self similarity property of a standard Brownian motion, and thus the fractal characteristics it has. Recall that a gaussian distribution is strictly characterized by its mean and covariance, hence, take note that in the following two lemmas, $B(t)$ and $X(t)$ both have mean zero and variance $t.$ 

\begin{lemma}
 If $B(t)$ is a standard Brownian motion, then $X(t) = \frac{1}{\sqrt{\lambda}} B(\lambda t)$ is a Brownian motion for any $\lambda > 0$ and $t \geq 0.$
\end{lemma}

\begin{proof}
For $X(t)$ to be a Brownian motion, it needs to be a Gaussian process and have the same mean and covariance structure as that of a Brownian motion. First, observe that $X(t)$ has a Gaussian distribution for any fixed $t \geq 0,$ since   $X(t) = \frac{1}{\sqrt{\lambda}} B(\lambda t).$ Secondly,
$$E[X(t)] = E[\frac{1}{\sqrt{\lambda}} B(\lambda t)] = \frac{1}{\sqrt{\lambda}} E[B(\lambda t)] = 0.$$
Moreover, for $0 < s < t$
\begin{align*}
E[X(t)X(s)] 
& = E[\frac{1}{\sqrt{\lambda}} B(\lambda t)\frac{1}{\sqrt{\lambda}} B(\lambda s)]\\
& =  \frac{1}{\lambda} E[B(\lambda t)B(\lambda s)]\\
& =  \frac{1}{\lambda} min(\lambda t, \lambda s) \\
& =  \frac{1}{\lambda} \cdot \lambda s = s.\\
\end{align*}
And thus, $X(t)$ is a Brownian motion.
\end{proof}

\begin{lemma}
If $B(t)$ is a standard Brownian motion, then so is the process 
\end{lemma}

\[  X(t) = 
     \begin{cases}
	0  &  t = 0 \\
	tB(1/t)  &  t > 0
    \end{cases}
\]

\vspace{3mm}

\begin{proof}
For fixed $t$, $X(t) = tB(1/t)$ clearly has a Gaussian distribution. Also
$$E[X(t)] = E[tB(1/t)] = tE[B(1/t)] = t \cdot 0 = 0.$$
Moreover, for $0<s<t$
\begin{align*}
E[X(t)X(s)] 
& = E[tB(1/t)sB(1/s)] \\
&= stE[B(1/t)B(1/s)] \\
&= st \cdot min(1/t, 1/s) \\
&= st \cdot 1/t = s.
\end{align*}
Hence, $X(t)$ is a Brownian motion.
\end{proof}

We can now show that Brownian motion is a martingale. Since by definition we have $E[B(t)] = 0,$ clearly $E[B(t)] < \infty$. Moreover, given a filtration $\{ \mathcal{F}_t : t \geq 0 \}$ and fixing $0 \leq s \leq t,$ we have $$E[B(t) | \mathcal{F}_s] = E[B(s) | \mathcal{F}_s] + E[B(t) - B(s) | \mathcal{F}_s] = B(s),$$
which is enough to show that Brownian motion is indeed a martingale. As we shall later see, the development of the {\Ito} integral takes advantage of the martingale property of Brownian motion. To also note, there are other Stochastic integrals that do not take advantage of this property (i.e. Stratonovich) which lead to different results. This paper focuses on the {\Ito} integral and hence on the properties of Brownian motion and the choices that lead to the {\Ito} integral.\\

\begin{mydef}
A process $\{X(t); 0 \leq t \leq T \}$ is said to be {\Holder} continuous of order $\alpha \in (0,1)$ if
$$P\bigg[  \displaystyle \sup \limits_{s,t \in [0,T]} \frac{|X(t) - X(s)|}{|t-s|^{\alpha}} \leq h \bigg] = 1,$$
where $h > 0$ is an appropriate constant.
\end{mydef}

By definition of Brownian motion, $B(t)$ is almost surely a continuous function. Moreover, sample paths of $B(t)$ can be shown to be {\Holder} continuous.  In [18], Breiman shows that every Brownian path is nowhere differentiable ([18]theorem 12.25) and consequently of unbounded variation in every interval ([18]corollary 12.27). Since Brownian motion has continuous paths, this means that at every point, there is a kink, $i.e.$ Brownian motion changes direction at every $t \geq 0!$

\begin{lemma}
$\langle B(t) \rangle_T = T,$ where $\langle B(t) \rangle_T$ is the quadratic variation of the Brownian motion $B(t)$ over the interval $[0, T].$
\end{lemma}

The proof of this lemma can be found in [9].\\ 

	For deterministic functions $f,$ such that $\int |f'(t)|^2 dt < \infty$, the quadratic variation $\langle f(t) \rangle_T$ is zero, which leads to the traditional definition of the classical Riemann integral.  But in this paper, our main concern  is characterizing a stochastic integral, where both $f(t,\omega)$ and $B(t)$ are random processes. As we shall see in the following section, both the non-zero quadratic variation and the martingale properties of Brownian motion are two of the motivating reasons for the construction of {\Ito} calculus, giving meaning to the {\Ito} integral.

\section{The {\Ito} Integral}\label{itogral}

For fixed interval $[S,T],$ let's look at the stochastic integral

\begin{align}
\label{itointegral}
I_{[S,T]} (f) = \int^T_S f(t, \omega)dB_t(\omega),
\end{align}
where $f(t,\omega)$ is a random function defined by the mapping $[0, \infty) \times \Omega \rightarrow \mathbb{R}$, restricted such that for any fixed $t$, the random variable $f(t,w)$ is $\mathcal{F}_t$-measurable and $B_t(\omega)$ is a Brownian motion. We include $\omega$ in the representation of $f$ and $B_t$ to show that indeed they are both random functions taking input values $\omega$ from a sample space $\Omega.$ Note that from now on, we may drop the $\omega$ from notation for simplicity and further will write in place of $B_t (\omega)$ either $B_t$ or $B(t).$ To continue, if $B(t)$  was differentiable, we could write $dB(t) = B'(t)dt.$ But, as we know from above, $B(t)$ is non-differentiable at every point, so the chain rule we are all accustomed to from ordinary calculus on deterministic functions does not apply. Moreover, the paths of $B(t)$ are not of bounded variation and hence \eqref{itointegral} is meaningless in the Stieltjes-Lebesgue sense. To make sense of \eqref{itointegral}, we begin by approximating $f(t,\omega)$ as a simple process, $i.e.$ $f(t,\omega)$ has the form  
$$\displaystyle \sum \limits_{j \geq 0} e_{j}(\omega) {\bf{1}}_{[t_j , t_{j+1}]} (t),$$
 where ${\bf{1}}_{[t_j , t_{j+1}]} (t)$ is the indicator function over the interval $[t_j , t_{j+1}].$\\

Let $\{ t_0, t_1,...,t_{n-1}, t_n\}$ be a partition of $[S,T]$ such that 
$$S = t_0 \leq t_1 \leq \cdots \leq t_{n-1} \leq t_n = T.$$ 

We want $f(t,\omega)$ to be constant, in $t,$ over each subinterval $[t_i,t_{i+1}].$ So, define  $f(t,\omega)$ to be $e_i (\omega)$ over each interval $[t_i,t_{i+1}].$ Thus $e_i (\omega)$ is a random variable  and independent of $t$ over each respective subinterval. Hence, it clearly follows that for each subinterval $[t_i,t_{i+1}],$

\begin{align*}
\int^{t_{i+1}}_{t_i} f(t, \omega)dB_t 
& = \int^{t_{i+1}}_{t_i} e_i (\omega) dB_t \\
& = e_i (\omega)[B(t_{i+1})-B(t_{i})],
\end{align*}

since $e_i (\omega)$ is independent of $t$. And hence,

\begin{align}
\label{itosum}
\int^T_S f(t, \omega)dB_t = \displaystyle \sum \limits_{i = 0}^{n-1} e_{i}(\omega) [B(t_{i + 1}) - B(t_{i})].
\end{align}

Note that in this case, we let $f(t,\omega)$ be a simple process, where we still need to make sense of $e_i (\omega)$. Remember that while constructing the Riemann integral over deterministic functions, we approximate the integrand over each subinterval by choosing any $t^{*}_i \in [t_i , t_{i+1}],$ our choice of $t^{*}_i$ is irrelevant. Then, as the length of the longest subinterval in the partition tends to zero, our limit turns out to be $\int^T_S f(t)dt.$ But, as will be shown, the choice of $t^{*}_i$ when the integrand is a random function, as in \eqref{itosum}, is important. \\

Let $I_L$ and $I_R$ be the left and right point approximation for \eqref{itosum}, respectively, where $e_i (\omega) = B(t^{*} _i).$ Thus,
\begin{align*}
E[I_L] 
& =  E \bigg[ \displaystyle \sum \limits_{i = 0}^{n-1}  B(t^{*} _i)[B(t_{i + 1}) - B(t_{i})] \bigg]\\
& = \displaystyle \sum \limits_{i = 0}^{n-1}  E \Big[ B(t_i)[B(t_{i + 1}) - B(t_{i})] \Big]\\
& =  \displaystyle \sum \limits_{i = 0}^{n-1}  E[B(t_i)] E[B(t_{i + 1}) - B(t_{i})]\\
& = 0,
\end{align*}
which follows from the definition of Brownian motion, $i.e.$ $B(t)$ has independent increments and the expectation of each increment is zero. On the other hand,
\begin{align*}
E[I_R] 
& = E \Bigg[ \displaystyle \sum \limits_{i = 0}^{n-1}  B(t^{*} _{i})[B(t_{i + 1}) - B(t_{i})] \bigg]\\
& =  \displaystyle \sum \limits_{i = 0}^{n-1}  E \bigg[ B(t_{i + 1})[B(t_{i + 1}) - B(t_{i})] \bigg]\\
& =  \displaystyle \sum \limits_{i = 0}^{n-1}  E \bigg[ [B(t_{i + 1})^2 - B(t_{i + 1})B(t_{i})] \bigg]\\
& =  \displaystyle \sum \limits_{i = 0}^{n-1}  \bigg[ E[B(t_{i + 1})^2] - E[B(t _{i + 1})B(t_{i})] \bigg]\\
& =  \displaystyle \sum \limits_{i = 0}^{n-1} [t_{i + 1} - t_{i}] = T,\\
\end{align*}
since the variance of $B(t)$ is equal to the length of the interval and the covariance of $B(t)$ and $B(s)$ is equal to the minimum of $s$ and $t.$ And so we see that the interpretation of \eqref{itointegral} depends greatly on the value of $t^{*}_i$ that is chosen. Moreover, $E[I_L] = 0$ and $E[I_R] = T$ reflects the fact that the variations of the paths of $B_t$ are too big for us to define the integral \eqref{itointegral} in the Riemann sense. In fact, as mentioned earlier, Breiman [18] shows that $B_t$ is almost surely nowhere differentiable, in particular the total variation of the path is almost surely infinite. Recall that we restricted $f(t,\omega)$ to be $\mathcal{F}_t$-measurable, so it seems reasonable enough to choose the approximating functions $e_i (\omega)$ to be $\mathcal{F}_t$-measurable as well. Since this is the case, we then choose $t^{*}_i$ to be the left end point, i.e. $t^{*}_i \in [t_i, t_{i+1}].$ This choice of $t^{*}_i$ leads to the Ito integral.\\

Notice that we have interpreted the {\Ito} integral, \eqref{itointegral} for simple functions $f(t,\omega),$ $i.e.$ functions of the type such that $f(t,\omega) = \displaystyle \sum \limits_{i = 0}^{n-1} e_i (\omega),$ where $e_i (\omega)$ is a random function constant in $t$ over the subinterval $[t_i , t_{i+1}].$ But, what we are really interested in, is what are all the types of functions $f(t,\omega)$ that the {\Ito} integral will be defined. 

\begin{mydef}
Let $V(S,T)$ be the class of functions $\{ f(t,\omega) : [0, \infty ) \times \Omega \rightarrow \mathbb{R} \}$ such that $f(t,\omega)$ is $\mathcal{F}_t$-adapted and $\int^T_S f(t, \omega)^2 dt < \infty.$  
\end{mydef}

It can be shown in [10] that for any function $f(t,\omega) \in V(S,T)$, there exists a sequence of simple functions $\varphi_n (t,\omega)$, such that the {\Ito} integral is defined for functions $f(t,\omega)$ as

\begin{align}
\label{itolimit}
\int^T_S f(t, \omega)dB_t =  \lim_{n \rightarrow \infty} \int^T_S \varphi_n (t,\omega) dB_t.
\end{align}

An outline of the approximation is given here. First, it can be shown that for any $f(t,\omega) \in V(S,T)$, there exists a sequence of bounded functions $\phi_n (t,\omega)$ such that 

\begin{align}
\label{bounded}
\lim_{n \rightarrow \infty} E \Bigg[ \int^T_S (f - \phi_n)^2 dt \Bigg] = 0.
\end{align}

Next, for any bounded function  $\phi (t,\omega) \in V(S,T)$, it can be shown that there exists a sequence of bounded continuous functions $\psi_n (t,\omega)$ such that

\begin{align}
\label{continuous}
\lim_{n \rightarrow \infty} E \Bigg[ \int^T_S ( \phi - \psi_n )^2 dt \Bigg] = 0.
\end{align}

And lastly, for any bounded continuous function $\psi (t,\omega) \in V(S,T),$ there exists a sequence of simple functions $\varphi_n (t,\omega)$ such that

\begin{align}
\label{simple}
\lim_{n \rightarrow \infty} E\Bigg[ \int^T_S ( \psi - \varphi_n)^2 dt \Bigg] = 0.
\end{align}

And so we can see, by combining together equations \eqref{bounded}, \eqref{continuous} and \eqref{simple} we get that

\begin{equation}
\label{itoconverge}
\lim_{n \rightarrow \infty} E\Bigg[ \int^T_S ( f - \varphi_n)^2 dt \Bigg] = 0.
\end{equation}

This is why in \eqref{itosum}, we initially chose $f(t,\omega)$ to be a simple function $e_i (\omega).$ In order to finish the discussion on why the limit in \eqref{itolimit} exists, let's first present some properties of the {\Ito} integral.

\vspace{5mm}

\begin{theorem}
Let $\varphi \in V(S,T)$ be a random process with continuous sample paths. Then the process $\int^T_S \varphi (t,\omega) dB_t$ is a martingale with respect to the filtration $\mathcal{F} = \{ \mathcal{F}_t : t \geq 0 \}.$
\end{theorem}

For the proof we refer to Oksendal in [10].\\ 

	Now, since \eqref{itointegral} is shown to be a martingale, we know it must be true that 
$$E \Bigg[ \int^T_S f(t, \omega)dB_t \Bigg] = 0$$
for all $t \geq 0.$ To evaluate the variance of \eqref{itointegral}, we need the result of the next theorem, known as the {\Ito} isometry.

\vspace{10mm}

\begin{theorem}
Let $f(t,\omega)$ and $B_t$ be defined as in \eqref{itointegral}. Then
\begin{align}
E\Bigg[ \Bigg( \int^T_S f(t, w)dB_t \Bigg)^2 \Bigg] = E \Bigg[ \int^T_S f(s, \omega)^2 ds \Bigg].
\end{align}
\end{theorem} 

\begin{proof}
\begin{align*}
E\Bigg[ \Bigg( \int^T_S f(t, \omega)dB_t \Bigg)^2 \Bigg]
& = E\Bigg[ \Bigg( \displaystyle \sum \limits_{i = 0}^{n-1} \int^{t_{i+1}}_{t_i} f(t, \omega)dB_t \Bigg)^2 \Bigg]\\
& =  E\Bigg[ \Bigg( \displaystyle \sum \limits_{i = 0}^{n-1} e_i(\omega) (B(t_{i+1}) - B(t_i)) \Bigg)^2 \Bigg]\\ 
& = E\Bigg[ \displaystyle \sum \limits_{i = j = 0}^{n-1} e_i(\omega) e_j(\omega) (B(t_{i+1}) - B(t_i))(B(t_{j+1}) - B(t_j)) \Bigg]\\ 
& =  \displaystyle \sum \limits_{i = 0}^{n-1} E\Bigg[ e_i(\omega)^2 (B(t_{i+1}) - B(t_i))^2 \Bigg]\\  
& + 2 \displaystyle \sum \limits_{i < j} E\Bigg[ e_i(\omega) e_j(\omega) (B(t_{i+1}) - B(t_i)) E[B(t_{j+1}) - B(t_j) | \mathcal{F}_t ] \Bigg]\\
& =  \displaystyle \sum \limits_{i = 0}^{n-1} E\Bigg[ e_i(\omega)^2 (B(t_{i+1}) - B(t_i))^2 \Bigg] + 0\\
& =  \displaystyle \sum \limits_{i = 0}^{n-1} E[e_i(\omega)^2] (t_{i+1} - t_i)\\
& = E\Bigg[  \int^T_S f(t, \omega)^2 dt \Bigg]. 
\end{align*}
This completes the proof.
\end{proof}

\vspace{10mm}

\begin{theorem}
Let $I_{[0,T]} (f)$ be the {\Ito} integral of $f(t,\omega)$ over the interval $[0,T].$ Then the quadratic variation of $I_{[0,T]} (f)$ is equal to $\int^T_0 f(t,\omega)^2 dt,$ $i.e.$ 
\begin{align}
\langle I_{[0,T]} (f) \rangle_T = \int^T_0 f(t,\omega)^2 dt.
\end{align}
\end{theorem}

\begin{proof}
Let $\{ t_0, t_1, ... , t_{n-1} , t_n \}$ be a partition of $[0 , T]$ such that $t_0 = 0$ and $t_n = T,$ and let $f(t,\omega) = e_i(\omega)$ for $t \in [t_i , t_{i+1}],$ as defined above. Also, to simplify notation, let $I_{[0,t]} (f) = I_t.$ First, observe that

\begin{align}
\label{itoqv}
\langle I_T \rangle_T = \displaystyle \sum \limits_{i = 0}^{n-1} [\langle I_{t_{i+1}}  \rangle_{t_{i+1}} -  \langle I_{t_i} \rangle_{t_{i}} ].
\end{align}

To compute \eqref{itoqv}, let $\{ s_0, s_1, ... , s_{m-1}, s_m \}$ be a partition of $[t_i , t_{i+1}]$ such that $s_0 = t_i$, $s_n = t_{i+1}$ and $M_m = \displaystyle \max \limits_{j} (s_{j+1} - s_j).$ Then, by definition 
\begin{align*}
\langle I_{t_{i+1}} \rangle_{t_{i+1}} -  \langle I_{t_i} \rangle_{t_{i}} 
& = \lim_{M_m \rightarrow 0} \displaystyle \sum \limits_{j = 0}^{m-1} [I_{s_{j+1}} - I_{s_j}]^2\\
& = \lim_{M_m \rightarrow 0} \displaystyle \sum \limits_{j = 0}^{m-1} [\int^{s_{j+1}}_{s_j} f(t_i) dB(t)]^2\\
& =  \lim_{M_m \rightarrow 0} \displaystyle \sum \limits_{j = 0}^{m-1} f(t_i)^2 [B(s_{j+1}) - B(s_j)]\\ 
& =  \lim_{M_m \rightarrow 0}  f(t_i)^2  \displaystyle \sum \limits_{j = 0}^{m-1} [B(s_{j+1}) - B(s_j)]^2\\
& = f(t_i)^2 (t_{i+1} - t_i).
\end{align*}
And hence it follows that
\begin{align*}
\langle I_T \rangle_T 
& = \lim_{M_n \rightarrow 0} \displaystyle \sum \limits_{i = 0}^{n-1} f(t_i)^2 (t_{i+1} - t_i)\\
& = \lim_{M_n \rightarrow 0} \displaystyle \sum \limits_{i = 0}^{n-1} \int^{t_{i+1}}_{t_i} f(u)^2 du\\
& =  \int^{T}_0 f(u)^2 du.
\end{align*}
\end{proof}

\vspace{5mm}

Let's return to the definition of the {\Ito} integral given in \eqref{itolimit} and give reason why the limit exists. By {\Ito}'s isometry
\begin{align*}
E[(I_{[S,T]} (\varphi_n) - I_{[S,T]} (\varphi_m) )^2] 
& = E[(I_{[S,T]} (\varphi_n - \varphi_m) )^2]\\
& = E \Bigg[ \int^T_S  (\varphi_n - \varphi_m)^2 dt \Bigg].\\
\end{align*}

Further, since $\displaystyle \lim_{n \rightarrow \infty} \varphi_n = f(t,\omega)$ and by the triangle inequality, we get
$$E \Bigg[ \int^T_S  (\varphi_n - \varphi_m)^2 dt \Bigg] \leq E \Bigg[ \int^T_S  (\varphi_n - f(t,\omega))^2 dt \Bigg] + E \Bigg[ \int^T_S  (f(t,\omega)- \varphi_m)^2 dt \Bigg],$$
which, by \eqref{itoconverge}, tends to zero as $n \rightarrow \infty.$ Hence the sequence $\big\{  \int^T_S \varphi_n dB(t) \big\}$ forms a Cauchy sequence in $L_2(\Omega, \mathcal{F} , \mathbb{P}),$ which is a complete space. Therefore the limit of  $\big\{  \int^T_S \varphi_n dB(t) \big\}$ exists and is an element of  $L_2(\Omega, \mathcal{F} , \mathbb{P}).$ The limit, by definition in \eqref{itolimit}, is the {\Ito} integral.

\vspace{5mm}

\begin{example}
Let's use the {\Ito} integral, as defined in \eqref{itolimit}, to calculate $\int^t_0 B_s dB_s.$ First, let $\{t_0 , t_1, ... , t_{n-1}, t_n \}$ be a partition of $[0,T]$ such that $t_0 = 0$ and $t_n = T.$ Now, define $\varphi_n (t,\omega) = B(t_i)$ whenever $t \in [t_i , t_{i+1}].$ Let us now check and see if  $\varphi_n (t,\omega)$ is an appropriate approximation to $f(t,\omega)$ as described above. Since the variation of a Brownian motion is equal to the length of the interval, we have

\begin{align*}
E \Bigg[ \int^T_0 ( \varphi_n - B(t))^2 dt \Bigg]
& = E \Bigg[ \displaystyle \sum \limits_{i = 0}^{n-1} \int^{t_{i+1}}_{t_i} ( B(t) - \varphi_n )^2 dt \Bigg]\\
& =  \displaystyle \sum \limits_{i = 0}^{n-1} \int^{t_{i+1}}_{t_i} E \Bigg[ ( B(t) - B(t_i) )^2 \Bigg] dt \\
& =  \displaystyle \sum \limits_{i = 0}^{n-1} \int^{t_{i+1}}_{t_i} (t - t_i) dt\\
& =  \displaystyle \sum \limits_{i = 0}^{n-1} \frac{1}{2} (t_{i+1} - t_i )^2.\\
\end{align*}
To continue, let $M_n$ be defined to be the maximum subinterval length between each $t_i$ in our partition, $i.e.$ $M_n = \displaystyle \max \limits_{i} (t_{i+1} - t_i).$ \\
And so for all $i \in \{0, 1, ... , n-1 \}$

$$\frac{1}{2} (t_{i+1} - t_i)^2   \leq   \frac{M_n}{2} (t_{i+1} - t_i). $$

Summing over all $i$ we get

$$\frac{M_n}{2} \displaystyle \sum \limits_{i = 0}^{n-1}  (t_{i+1} - t_i)  =  \frac{M_n}{2} T.$$

But, as $n \rightarrow \infty,$ $M_n \rightarrow 0.$ So we get 
$$E [ \int^T_0 ( \varphi_n - B(t))^2 dt] \rightarrow 0$$
 as $n \rightarrow \infty.$ Hence, $\varphi_n (t,\omega)$ is an appropriate approximation for $B(t),$ thus 
$$ \int^T_0 B(t) dB(t) = \lim_{n \rightarrow \infty}\int^T_0 \varphi_n (t,\omega) dB(t).$$ 
And so

\begin{align*}
\int^T_0 \varphi_n (t,\omega) dB(t) 
& = \displaystyle \sum \limits_{i = 0}^{n-1} B(t_i)[B(t_{i+1}) - B(t_i)]\\
& = \displaystyle \sum \limits_{i = 0}^{n-1}  \frac{1}{2}[B(t_{i+1})^2 - B(t_{i})^2 - (B(t_{i+1}) - B(t_{i}))^2]\\
& =   \frac{1}{2} B(T)^2 - \frac{1}{2} \displaystyle \sum \limits_{i = 0}^{n-1}  (B(t_{i+1}) - B(t_{i}))^2\\
\end{align*}

But, as $n \rightarrow \infty,$ observe that $\displaystyle \sum \limits_{i = 0}^{n-1}  (B(t_{i+1}) - B(t_{i}))^2 \rightarrow T.$ And so we have 
$$ \int^T_0 B(t) dB(t) = \frac{1}{2} B(T)^2 - \frac{1}{2} T.$$
\end{example}

 If $B(t)$ was differentiable, we would simply get that $\int^T_0 B(t) dB(t) = \frac{1}{2} B(T)^2,$ But it's not. So where does the $\frac{1}{2} T$ term come from? It comes from the fact that Brownian motion has non-zero quadratic variation. Remembering that the {\Ito} integral is a martingale, we see that this extra term makes sense because
 $$E \Bigg[ \int^T_0 B(t) dB(t) \Bigg] = 0.$$
 But
 $$E \Bigg[ \frac{1}{2} B(T)^2 \Bigg] = \frac{1}{2} T.$$

\subsection{The {\Ito} Formula}\label{itormula}

As shown by the example in the previous section, the procedure of computing an {\Ito} integral by finding an approximating sequence satisfying \eqref{itolimit} can be rather time and work consuming. When calculating the classical Riemann integral, which is defined as the limit of Riemann sums, one takes advantage of the fundamental theorem of calculus and the chain rule, making calculations much easier. In the same vein, it is also desirable to have an {\Ito} integral version of the chain rule. This version is called the {\Ito} formula. 

\begin{mydef}
An {\Ito} process is a stochastic process $X(t)$ on $(\Omega, \mathcal{F} , \mathbb{P})$ of the form
\begin{align}
X(t) = X(0) + \int^t_0 \mu (s,\omega) ds +  \int^t_0 \nu (s,\omega) dB(s). 
\end{align}
Where $\nu \in V[S,T],$ $\mu$ is $\mathcal{F}_t$ adapted and that $\int^t_0 \nu^2 ds < \infty$ and  $\int^t_0 |\mu| ds < \infty$ for all $t \geq 0,$ almost surely. 
\end{mydef}

\vspace{5mm}

\begin{theorem}
({\Ito} Formula) Let $X(t)$ be an {\Ito} process. Let $f(t, x)$ be such that $f_t(t, x)$, $f_x(t, x)$ and $f_{xx}(t, x)$ exist and are continuous. Then, $Y(t) = f(t, X(t))$ is an {\Ito} process and 
\begin{align*}
Y(T) - Y(0) = \int^T_0 f_t (t, X(t)) dt +  \int^T_0 f_x ( t, X(t)) dX(t)
 + \frac{1}{2} \int^T_0 f_{xx}(t, X(t)) (dX(t))^2.
\end{align*}
Where $(dX(t))^2 = dX(t) \cdot dX(t)$ is calculated according to the rules 
$$dt \cdot dt = dt \cdot dX(t) = dX(t) \cdot dt = 0$$
and 
$$dX(t) \cdot dX(t) = dt.$$
\end{theorem}

Furthermore, we can see that if $X(t) = B(t),$ we get that
\begin{align*}
Y(T) - Y(0) = \int^T_0 f_t (t, B_t) dt +  \int^T_0 f_x ( t, B_t) dB_t
 + \frac{1}{2} \int^T_0 f_{xx}(t, B_t) dt.
\end{align*}

The proof of this can be found in [10] from the Taylor series expansion of $f(t,x).$
\vspace{5mm}

\begin{example}
Let $f(t,x) = \frac{1}{2}x^2$ and $B(t)$ be a standard Brownian motion. Then by the {\Ito} formula,
\begin{align*}
df(t, B_t) 
&= f_t dt + f_x dB_t + \frac{1}{2} f_{xx} dt\\
&= 0 \cdot dt + B_t dB_t + \frac{1}{2} \cdot 1 dt\\
&= B_t dB_t + \frac{1}{2} dt.
\end{align*}
By integrating we get
\begin{align*}
\int^t_0 d\Big(f(t,B_t) \Big)
& = \int^t_0 d \Big( \frac{1}{2} B^s_2 \Big)\\
& = \frac{1}{2} B^2(t) - \frac{1}{2} B^2(0)\\
&= \frac{1}{2} B^2(t),
\end{align*}
since $B(0) = 0.$
Also,
\begin{align*}
\int^t_0 \Big( B_s dB_s + \frac{1}{2} ds \Big)
&= \int^t_0 B_s dB_s + \frac{t}{2}.
\end{align*}
Hence,
$$\int^t_0 B_s dB_s =  \frac{1}{2} B^2(t) - \frac{t}{2}.$$
\end{example}

\section{Fractional Brownian Motion}\label{FBM}
We now begin our discussion of fBm, as introduced in the Introduction. This process was first introduced by Kolmogorov in [5].  The Hurst index was named by Mandelbrot from the statistical analysis of hydrologist Harold Edwin Hurst, who studied yearly water run-offs of the Nile River over the years 662 to 1469 in [8]. As mentioned, fBm is a generalization of Brownian motion. But, as will be presented, fBm, whenever $H \neq 1/2,$ behaves very differently than Brownian motion (when $H = 1/2$). There are two properties of importance in which fBm with $H \neq 1/2$ differs from Brownian motion ($H = 1/2$), fBm does not have independent increments and it is not a martingale (even more, a semi-martingale!). These characteristics inherent in Brownian motion lead to the construction of the {\Ito} integral.  Hence, as we develop the framework for the construction of the stochastic integral w.r.t. fBm, we must approach it differently. The difficulty in this is largely due to the fact that fBm fails to be a semi-martingale (as will be shown). Which, as discussed in [12], [13] and [15], presents the major issue of developing a stochastic integral of this type, in that "reasonable" stochastic integration is possible only w.r.t. semi-martingales. Since fBm is not a semi-martingale, it can be expected that stochastic integrals w.r.t. this motion are not continuous. To begin this construction, we first introduce a formal definition of fBm.

\begin{mydef}\label{fbmdef}
A Gaussian Process $B^{(H)}_t = \{B^{(H)}(t), t \geq 0 \}$ is called a fractional Brownian motion (fBm) of Hurst index $H \in (0,1)$ if it has mean zero with covariance function
 
\begin{equation}
\label{fBmcov}
R_H (t,s) = E[B^{(H)}(t)B^{(H)}(s)] = \frac{1}{2}(t^{2H} + s^{2H} - |t - s|^{2H}).
\end{equation}
\end{mydef}

\vspace{5mm}

Similar to Brownian motion, fBm is, by definition, a Gaussian process, and therefore it is strictly characterized by its mean and covariance. Its mean by definition is zero and covariance given by $R_H (t,s).$ Hence the following three properties are obtained through $R_H (t,s)$ in definition \ref{fbmdef}.

\begin{enumerate}
	\item Self-similarity: The process $\{a^{-H}B^{H}(at), t \geq 0 \}$ has the same law as $\{B^{H}(t), t \geq 0 \},$ i.e. $a^{-H}B^{H}(at) \sim B^{H}(t).$
	\item Stationary increments: $B^{H}(t + s) - B^{H}(s) \sim B^{H}(t)$  for $s,t \geq 0.$
	\item Variance: $E[B^{H}(t)^2] = t^{2H}$ for all $t \geq 0.$ 
\end{enumerate}

\begin{proof}
\begin{enumerate}
	\item By definition $E[B^H (t)] = 0$ and hence $E[a^{-H}B^{(H)}(at)] = 0$ also. Thus to show that both processes have the same probability distribution, it is sufficient enough to show that they both have the same covariance. Let $a > 0$ and $s,t \geq 0.$ 
	$$E[a^{-H}B^{H}(at)a^{-H}B^{H}(as)]$$
	  $$=  a^{-2H}E[B^{H}(at)B^{H}(as)]$$
	 $$= \frac{1}{2}a^{-2H}[(at)^{2H} + (as)^{2H} - |at - as|^{2H}]$$
	  $$= \frac{1}{2}a^{-2H}a^{2H}[(t)^{2H} + (s)^{2H} - |t - s|^{2H}]$$
	 $$= \frac{1}{2}[t^{2H} + s^{2H} - |t - s|^{2H}]$$
	 $$= E[B^{H}(t)B^{H}(s)].$$
And thus fBm is self-similar.

	\item Again, since $E[B^{(H)}(t + s) - B^{(H)}(t)]$ is clearly zero, it suffices to show that both processes have equal covariance. Let $r,s,t, \geq 0,$ then\\
	
	$E[(B^{H}(t + s) - B^{H}(s))(B^{H}(r + s) - B^{H}(s))]$
	\begin{align*}
	& = E[(B^{H}(t + s)(B^{H}(r + s)] - E[(B^{H}(t + s)B^H (s)]\\
	& -E[B^H (s)B^{H}(r + s)] + E[B^H (s)B^H(s)]\\
	& = \frac{1}{2} \Big\{ (t+s)^{2H} + (r+s)^{2H} - |t-r|^{2H} - \big[ (t+s)^{2H} + s^{2H} - t^{2H} \big]\\  	& - \big[ s^{2H} + (r+s)^{2H} - r^{2H} \big] +  \big[ s^{2H} + s^{2H} \big] \Big\}\\
	& = \frac{1}{2} \Big\{ t^{2H} + s^{2H} - |t-r|^{2H} \Big\}\\
	& = E[B^H(t)B^H(r)].
	\end{align*}
Hence fBm has stationary increments.

	\item $E[B^{H}(t)^2] = E[B^{H}(t)B^H(t)] = \frac{1}{2}(t^{2H} + t^{2H} - |t-t|^{2H}) = t^{2H}.$
\end{enumerate}
\end{proof}

From the self-similarity property of fBm, we get that $B^H(0) = 0$ almost surely. This is shown by taking advantage of that fact that for any $a > 0$ we get that $a0 = 0.$ Hence
$$B^H(0) = B^H(a0).$$
But by property 1, 
$$B^H(a0) \sim a^{-H}B^H(0)$$
 (by $\sim$ we mean they have the same law). Moreover 
 $$B^H(0) \sim a^{-H}B^H(0),$$
 It follows that since this is true for all $a > 0,$ $B^H(0) = 0$ a.s.\\
 
 Now we know by definition that the mean of fBm is zero. But let's take advantage of the self-similarity of fBm and that it has stationary increments to explicitly show that indeed $E[B^H(t)]=0$ for all $t \geq 0.$ To show this, we first use the self-similarity of fBm,
 $$E[B^H(2t)] = 2^H E[B^H(t)].$$
 However, since $B^H(t)$ has stationary increments,
 $$E[B^H(2t)] = E[B^H(2t) - B^H(t)] + E[B^H(t)].$$
 Hence, 
 $$2^H E[B^H(t)] = 2E[B^H(t)].$$
 But $H \in (0,1),$ which means $E[B^H(t)]$ must be zero.\\
 
 Independence of Brownian motion paths is an important property that distinguishes itself from fBm, which as it turns out, does not have independent increments. To show this we assume that fBm does have independent increments. In doing so, we will see that the only way for this to be possible is for $H= 1/2.$ To continue, we assume fBm has independent increments, hence it would be the case for $0 < s < t$ that
 $$E[B^H(s) (B^H(t) - B^H(s))] = E[B^H(s)] E[B^H(t) - B^H(s)] = 0, $$
 since $B^H(t)$ has mean zero.  Moreover, by definition \ref{fbmdef} we get
 \begin{align*}
 E\big[B^H(s) (B^H(t) - B^H(s))\big] 
 & = E\big[B^H(s)B^H(t)\big] - E\big[B^H(s)B^H(s)\big]\\
 & = \frac{1}{2} \big( t^{2H} + s^{2H} - (t-s)^{2H} - 2s^{2H} \big)\\
 & = \frac{1}{2} \big( t^{2H} - s^{2H} - (t-s)^{2H} \big).
 \end{align*}
 But then this means that 
 $$\frac{1}{2} \big( t^{2H} - s^{2H} - (t-s)^{2H} \big) = 0.$$
 However, since $s,t \neq 0,$ we get that this can only be true when $H = 1/2.$
 Therefore we have shown that whenever $H \neq 1/2,$ that $B^H(t)$ in fact does not have independent increments.
 
 Let's now explore the correlation between the increments of fBm and use this to present the long-range dependence property of fBm. From \eqref{fBmcov}, it is easily shown that\\
 
 $E\Big[\big(B^H(t) - B^H(s)\big)\big(B^H(u) - B^H(v)\big)\Big]$
 
 \begin{align}
 \label{corr}
 = \frac{1}{2} \big( |s-u|^{2H} + |t-v|^{2H} - |t-u|^{2H} - |s-v|^{2H} \big).
 \end{align}
 
 For $H \in (0,1/2) \cup (1/2,1),$ $\alpha = H - 1/2$ and $t_1 < t_2 < t_3 < t_4,$ it follows from \eqref{corr} that \\
 
 $E\Big[\big(B^H(t_4) - B^H(t_3)\big)\big(B^H(t_1) - B^H(t_1)\big)\Big]$
\begin{align*} 
& =  \frac{1}{2} \big[(t_3-t_2)^{2H} - (t_3-t_1)^{2H} - (t_4-t_2)^{2H} + (t_4-t_1)^{2H}\big]\\
& = \frac{1}{2} \big[(t_3-t_2)^{2H} - (t_3-t_1)^{2H}\big] -\frac{1}{2}(t_4-t_2)^{2H} - (t_4-t_1)^{2H}\big]\\
& = H \int^{t_2}_{t_1} (t_4 - v)^{2H-1} dv - H \int^{t_2}_{t_1} (t_3 - v)^{2H-1} dv\\
& = (2H-1)H \int^{t_2}_{t_1} \int^{t_4}_{t_3} (u - v)^{2H-2} dudv\\
& = 2\alpha H \int^{t_2}_{t_1} \int^{t_4}_{t_3} (u-v)^{2\alpha -1} du dv.
 \end{align*}
 Hence, $\alpha > 0$ whenever $H \in (1/2, 1),$ and so increments of fBm are positively correlated. Furthermore, $\alpha < 0$ whenever $H \in (0, 1/2),$ thus increments of fBm are negatively correlated.\\
 
 Define $X(n) = B^H(n+1) - B^H(n),$ $n \geq 1.$ Then clearly $X(n)$ is a Gaussian stationary sequence with unit variance. Moreover the covariance function of $X(t)$ is 
 \begin{align*}
 r_H (n) 
 & = E[X(0) X(n)]\\
 & = \frac{1}{2} \big( (n+1)^{2H} - 2n^{2H} + (n-1)^{2H} \big).
 \end{align*}
 
 If $H=1/2$ then we get that $r(n) = 0$ implying that the increments of $X(n)$ are uncorrelated. But, if $H \neq 1/2,$ we get that as $n$ tends to infinity
 $$r_H (n) \sim H(2H-1)n^{2H-2}.$$
Thus we get 
\begin{enumerate}
	\item If $0 < H < \frac{1}{2}$ then $\displaystyle \sum \limits_{n = 0}^{\infty} |r_H(n)| < \infty.$
	\item If $H = \frac{1}{2}$ then \{X(n)\} is uncorrelated.
	\item If $\frac{1}{2} < H < 1$ then $\displaystyle \sum \limits_{n = 0}^{\infty} |r_H(n)| = \infty.$
\end{enumerate} 
Whenever a process has this property, when  $\displaystyle \sum \limits_{n = 0}^{\infty} |r_H(n)| = \infty,$ as in case $3,$ we say that it has long-range dependence.

\begin{proposition}
Let $H \in (0, 1).$Sample paths of $B^H(t)$ are not differentiable.
\end{proposition}

\begin{proof}
Recall that the definition of the derivative of a function $f(x)$ at $x_0$ is 
\begin{align*}
\displaystyle \lim \limits_{h \rightarrow 0} \frac{f(x_0 + h) - f(x_0)}{h}.
\end{align*}

And that $f(x)$ is differentiable at $x_0$ if this limit exists. But by the self similarity property of $B^H(t),$ $B^H(t + h) - B^H(t)$ has the same law as $B^H(h).$ Hence, consider the even
\begin{align*}
A(t, \omega) 
& = \bigg\{ \displaystyle \sup_{0 \leq s \leq t} \bigg| \frac{B^H(h)}{h} \bigg| > d \bigg\}.
\end{align*}

Then, we have that 
$$A(t_{n+1}, \omega) \subseteq A(t_n , \omega)$$
for a decreasing sequence $\{ t_n \}$ to zero. Moreover,
$$A(t_n , \omega) \supseteq \Big( \bigg| \frac{B^H(t_n)}{t_n} \bigg| > d \Big) = \big( |B^H(1)| > t^{1-H}_n d \big).$$
But, 
$$\displaystyle \lim \limits_{n \rightarrow \infty} P\big[|B^H(1)| > t^{1-H}_n d \big] = 1.$$
Since this is true for any $d,$ it must be the case that the derivative does not exist at any point along any sample path of $B^H(t).$

\end{proof}
 
When two random processes $\{X(t)\}$ and $\{Y(t)\}$ satisfy $P[ X(t) = Y(t)$ for all $t \geq 0] = 1$,
we say that one is a modification of the other. It is known that Brownian motion has a modification, the sample paths of which, as discussed in earlier sections, are {\Holder} continuous almost surely, but, these sample paths are nowhere differentiable. It turns out that this is also true for fractional Brownian motion for any value of $H \in (0,1)$ as well. To classify just how continuous fBm is, we give the following Lemma, known as A general version of Kolmogorov's criterion.  We will then use this to show that fBm has a modification, in which the sample paths are {\Holder} continuous.

\begin{lemma}\label{kolcrit}
If a stochastic process $\{ X(t) \}$ satisfies
\begin{align}
E[|X(t) - X(s)|^{\delta}] \leq C|t-s|^{1 + \epsilon},
\end{align}
for all $t,s$ and for some $\delta > 0$, $\epsilon > 0$ and $C > 0,$ then $\{ X(t) \}$ has a modification, the sample paths of which are {\Holder} continuous of order $\gamma \in [0, \epsilon/\delta ).$
\end{lemma}

A proof can be found in [25].

\begin{theorem}
Fractional Brownian motion $\{ B^H(t)\}$ has a modification, the sample paths of which are {\Holder} continuous of order $\beta \in [0, H ).$ 
\end{theorem}

\begin{proof}
Let $0 < \gamma < H.$ Then it follows from the self-similarity of fBm and since it has stationary increments that 
\begin{align*}
E\big[ |B^H(t) - B^H(s)|^{1/\gamma} \big] 
& = E\big[ |B^H(|t-s|)|^{1/\gamma} \big]\\
& = |t-s|^{H/\gamma}E\big[ |B^H(1)|^{1/\gamma} \big]. 
\end{align*}
Thus we have satisfied the conditions in Lemma \ref{kolcrit} where $\delta = 1/\gamma$ and $1 + \epsilon = H/\gamma.$ Moreover, by Kolmogorov's criterion, $B^H(t)$ has a modification that is {\Holder} continuous of order $\beta \in [0, H - \gamma).$ But, since $\gamma$ can be as small as we like, it follows that $\beta \in [0, H).$
\end{proof}

The definition of the {\Ito} integral is a direct consequence of the martingale property of Brownian motion. But fBm does not exhibit this property, in fact, fBm is not even a semi-martingale. There are many different proofs revealing this fact (a rather nice one is given in [13]). We state the theorem and present a simple proof here. But first, we need to find the $p$-variation of $B^H(t).$

\begin{mydef}
The $p$-variation of a random process $X(t)$ over the interval $[0,T]$ is 

\begin{align}
\label{pv}
V_p(X, [0,T]) = \displaystyle \sup \limits_{\pi}  \displaystyle \sum \limits_{i = 1}^{n} |X(t_i) - X(t_{i-1})|^p, 
\end{align}  

where $\pi$ is a finite partition over $[0,T].$ The index of the $p$-variation is defined to be

\begin{align}
\label{ipv}
I(X,[0,T]) = inf\big\{p > 0; V_p(X, [0,T]) < \infty \big\}.
\end{align}
\end{mydef}

\begin{lemma}
$I(B^H(t),[0,T]) = \frac{1}{H}.$ Moreover, $V_p(B^H(t), [0,T]) = 0$ when $pH>1$ and $V_p(B^H(t), [0,T]) = \infty$ when  $pH<1.$
\end{lemma}

A proof can be found in [1].

This can be seen when we take into consideration that 
$$E[|B^H(t_i) - B^H(t_{i-1})|^p] = E[|B^H(1)|^p] |t_i - t_{i-1}|^{pH},$$
and plugging this into \eqref{pv} and applying \eqref{ipv}.

\begin{theorem} \label{fbmnotsemi}
$\{B^H(t): t \geq 0\},$ for $H \neq 1/2,$ is not semimartingale.
\end{theorem}	

\begin{proof}
A process $\{X(t), \mathcal{F}_t, t \geq 0\}$ is called a semimartingale if it admits the Doob-Meyer decomposition $X(t) = X(0) + M(t) + A(t),$ where $M(t)$ is an $\mathcal{F}_t$ local martingale with $M(0) = 0$, $A(t)$ is a $\it{c\grave{a}dl\grave{a}g}$ adapted process of locally bounded variation and $X(0)$ is $\mathcal{F}_0$-measurable.  Moreover, any semimartingale has locally bounded quadratic variation $[11].$ 

	Now, let $X(t) = B^H(t).$ If $H \in (0, 1/2),$ then $B^H(t)$ cannot even be a martingale since it has infinite quadratic variation, hence, it is not a semimartingale.
	
	If $H \in (1/2, 1)$ then the quadratic variation of $B^H(t)$ is zero. So, let's suppose that it is a semimartingale. Then, $M(t) = B^H(t) - A(t)$ has quadratic variation equal to zero. So, from $[11],$ $M(t) = 0$ for all $t$ $a.s.$  Then that would mean that $B^H(t) = A(t),$ but this can't be the case since $B^H(t)$ has unbounded variation. Hence $B^H(t)$ is not a semimartingale for any $H \neq 1/2.$ 
\end{proof}

Now we show that the fractional Brownian motion can be represented as a stochastic integral. Consider,

\begin{align*}
Z(t) 
&=  \frac{1}{C(H)} \int_{\mathbb{R}} \bigg((t - s)^{H - \frac{1}{2}}_+ - (-s)^{H- \frac{1}{2}}_+ \bigg) dB(s)\\
&= \frac{1}{C(H)} \Bigg( \int^0_{- \infty} ((t - s)^{H - \frac{1}{2}} - (-s)^{H- \frac{1}{2}} ) dB(s) +  \int^{t}_0 ((t - s)^{H - \frac{1}{2}}dB(s) \Bigg)\\
\end{align*}

where $B(t)$ is a standard Brownian motion and 
$$C(H) = \Bigg( \int^0_{- \infty} \bigg[(1 - s)^{H - \frac{1}{2}} - (-s)^{H- \frac{1}{2}} \bigg]^{2} ds + \frac{1}{2H} \Bigg)^{\frac{1}{2}}$$

\begin{proof}

First, notice that $Z(t)$ is a stochastic integral with respect to a standard Brownian motion like that in \eqref{itointegral}, where $f$ is a deterministic function such that $\int |f(x)|^2dx < \infty.$ Hence it must be Gaussian with $E[Z(t)] = 0.$ Moreover, since $Z(t)$ is a Gaussian, we know that it must be strictly characterized by its mean and covariance. Hence, let's show that the covariance of $Z(t)$ is indeed the same given in \eqref{fBmcov}.
Observe that 

\begin{align}
\label{fBmexpand}
E[|Z(t) - Z(s)|^2]  = E[Z(t)^2]  - 2E[Z(t)Z(s)]  + E[Z(s)^2].
\end{align}

To continue we use the property that if $\int_A |f(x)|^2dx < \infty$, $A \subset \mathbb{R},$ then

\begin{align}
\label{fBmisom}
E\Big[\big(\int_A f(x)dB(x)\big)^2\Big] = \int_A f(x)^2dx.
\end{align}

Thus, by \eqref{fBmisom}, we have
\begin{align*}
E[Z(t)^2] 
& =  \frac{1}{C(H)^2} \int_{\mathbb{R}} \Bigg[(t - s)^{H - \frac{1}{2}}_+ - (-s)^{H- \frac{1}{2}}_+ \Bigg]^2 ds\\
& =  \frac{1}{C(H)^2}  t^{2H}\int_{\mathbb{R}} \Bigg[(1 - u)^{H - \frac{1}{2}}_+ - (-u)^{H- \frac{1}{2}}_+ \Bigg]^2 du\\
&= t^{2H}
\end{align*}
where $s=tu.$ Similarly,

\begin{align*}
E[|Z(t) - Z(s)|^2] 
& =  \frac{1}{C(H)^2} \int_{\mathbb{R}} \Bigg[(t - u)^{H - \frac{1}{2}}_+ - (s-u)^{H- \frac{1}{2}}_+ \Bigg]^2 ds\\
& = \frac{1}{C(H)^2} t^{2H}\int_{\mathbb{R}} \Bigg[(t-s-u)^{H - \frac{1}{2}}_+ - (-u)^{H- \frac{1}{2}}_+ \Bigg]^2 du\\
&= |t-s|^{2H} 
\end{align*}

Hence, by \eqref{fBmexpand}, we get

\begin{align*}
E[Z(t)Z(s)] 
& = -\frac{1}{2} \bigg(E[|Z(t) - Z(s)|^2]  - E[Z(t)^2]  - E[Z(s)^2] \bigg)\\ 
& = \frac{1}{2}(t^{2H} + s^{2H} - |t - s|^{2H}).\\
\end{align*}

Hence, $Z(t)$ is a fractional Brownian motion with Hurst index $H$.
\end{proof}

Let's continue the discussion of representing $B^H(t)$ as a stochastic integral and show how fBm is directly related to the fractional calculus we developed earlier. It was just shown above that if $B(t)$ is a  standard Brownian motion on $\mathbb{R},$ then 

\begin{align}
\label{fBmintegral}
Z(t) = \int^{\infty}_{- \infty} f(t,s) dB(s)
\end{align}
is a fractional Brownian motion with $f(t,s)$ defined as above. Let $H \in ( 1/2 , 1)$ and $x < 0 < t,$ then, for all $t \in \mathbb{R},$ we claim that 

\begin{align}
\label{fBmfrac}
 \bigg((t - s)^{H - \frac{1}{2}}_+ - (-s)^{H- \frac{1}{2}}_+ \bigg)
 & = \Gamma(H + 1/2) \Big( I^{H - \frac{1}{2}}_{-} {\bf{1}}_{[0,t]} \Big) (s),
 \end{align}
 
where ${\bf{1}}_{[0,t]}$ is the indicator function.

\begin{align*}
\Longrightarrow  \Big( I^{H - \frac{1}{2}}_{-} {\bf{1}}_{[0,t]} \Big) (s)
& = \frac{1}{\Gamma(H - 1/2)} \int^{\infty}_{x}  {\bf{1}}_{[0,t]} (u) (u-s)^{H- 3/2} du\\
& = \frac{1}{\Gamma(H - 1/2)} \int^{t}_{0} (u-s)^{H- 3/2} du\\ 
& =  \frac{1}{\Gamma(H + 1/2)} \Big[ (t - s)^{H - 1/2} - (-s)^{H - 1/2} \Big]
\end{align*}

Thus, plugging \eqref{fBmfrac} into \eqref{fBmintegral}, we get

\begin{align*}
Z(t) 
& = \frac{\Gamma(H + 1/2)}{C(H)} \int^{\infty}_{- \infty} \Big( I^{H - \frac{1}{2}}_{-} {\bf{1}}_{[0,t]} \Big) (s) dB(s)\\
\end{align*}

We now continue to informally show how fBm ($H \neq 1/2$) is related to Brownian motion ($H = 1/2$). It is important to remark that we will proceed in a non-rigorous way. We denote the derivative of the standard Brownian motion (a.k.a. white noise) as $B'(t),$ even though this derivative does not exist. In doing so, we can heuristically show how fBm is related to Brownian motion.  In light of this, by the fractional integration by parts formula \eqref{fracibp}, we get
\begin{align*}
Z(t) 
& = \frac{\Gamma(H + 1/2)}{C(H)} \int^{\infty}_{- \infty} \Big( I^{H - \frac{1}{2}}_{-} {\bf{1}}_{[0,t]} \Big) (s) dB(s)\\
& = \frac{\Gamma(H + 1/2)}{C(H)} \int^{\infty}_{- \infty} {\bf{1}}_{[0,t]} (s) \Big( I^{H - \frac{1}{2}}_{+} B'\Big)(s) ds\\
& = \frac{\Gamma(H + 1/2)}{C(H)} \int^t_0 \Big( I^{H - \frac{1}{2}}_{+} B'\Big)(s) ds\\
& = \frac{\Gamma(H + 1/2)}{C(H)} \Big( I^{H + \frac{1}{2}}_{+} B'\Big)(t). 
\end{align*}
Furthermore, by using the identity property of fractional calculus, we have
\begin{align}\label{fbmrelbm}
 \big( D^{H + \frac{1}{2}}_{+} Z\big)(t) 
 & =  \frac{\Gamma(H + 1/2)}{C(H)} B'(t).
\end{align}
The expression given in \eqref{fbmrelbm} is then meaningless, considering $B'(t)$ does not exist. Although, recall that from Section \eqref{fraccalc}, we can write $B'(t)$ as $\big(D^1_+ B\big)(t)$ and hence can express \eqref{fbmrelbm} as
\begin{align}\label{fbmbmmeaning}
Z(t)
& = \frac{\Gamma(H + 1/2)}{C(H)} \big( I^{H - \frac{1}{2}}_{+} B\big)(t).
\end{align}
This means that $Z(t)$ can be obtained by integrating $B(t)$ over the real line $H - 1/2$ times. Recall that the expression given in \eqref{fbmbmmeaning} was derived by assuming the existence of $B'(t),$ which in actuality does not exist. Currently there is no theory giving meaning to $B'(t)$ and hence making \eqref{fbmbmmeaning} meaningless. In conclusion, assuming that we can somehow give some formal meaning to $B'(t),$ a fractional Brownian motion with Hurst index $H$ can be obtained by integrating $B(t)$ over the real line $H - 1/2$ times.

\section{Pathwise Integration for Fractional Brownian Motion}\label{fbmstochcalc}

Now that we have developed some elementary properties of fBm, our next aim is to define stochastic integrals of the form

\begin{align}
\label{fBmsi}
\int^T_0 f(t, \omega ) dB^H(t),
\end{align}
where $B^H(t)$ is a fractional Brownian motion and $f(t, \omega )$ is a stochastic process.  When we constructed the {\Ito} integral with respect to Brownian motion, we took advantage of its martingale property. But as we have shown, fBm does not exhibit this property, which by the {\it{Bichteler-Dellacherie}} {\it{Theorem}}, aids strongly against the construction of an integral with respect to fBm.  

\begin{mydef} \label{goodintegrator}
Let $\mathcal{S}$ be the vector space of simple stochastic integrands.  A real-valued c\`{a}dl\`{a}g, adapted process $\{X(t) : t \in [0, T] \}$ is called a good integrator if the integration operator, $I_X : \mathcal{S} \rightarrow L^0(\mathbb{P})$, is continuous, where $L^0(\mathbb{P})$ is the space of all random variables, with the metrizable topology of convergence in probability.
\end{mydef}

In other words, if we take an element $H$ from $\mathcal{S}$, which is of the form
\begin{align*}
H 
& =  \displaystyle \sum \limits^k_{i = 1} H_i \bf{1}_{(t_i , t_{i + 1}]},
\end{align*}
where $k$ is finite and $H_i$ are bounded $\mathcal{F}_{t_i}$-measurable random variables, then by $I_X$ acting on $H$ we get the random variable
\begin{align*}
I_X (H)
& = \displaystyle \sum \limits^k_{i = 1} H_i (X_{t_{i+1}} - X_{t_i}).
\end{align*}
Moreover, for $I_X$ to be continuous, it must be the case that given $\mathbb{P}$ a.e. $S$ and for all $H \in \mathcal{S}$, if $H_n \rightarrow H$, then $I_X(H_N) \rightarrow I_X(H)$ in probability. Hence, by definition \ref{goodintegrator}, good integrators yield a continuous integration operator, but this begs the question as to how one can characterize a good integer in order to insure that this is true.

\begin{theorem} \label{BD}
(Bichteler-Dellacherie) For a real-valued, c\`{a}dl\`{a}g, adapted process $\{X(t) : t \in [0, T] \}$ the following are equivalent:
\begin{enumerate}
\item X(t) is a good integrator.
\item X(t) is a semimartingale. 
\end{enumerate}
\end{theorem}

This theorem explicitly characterizes good integrators as those being semimartingales, thus creating the major difficulty in constructing such an integral with fBm as the integrator.  For further discussion on good integrators and the proof of the {\it{Bichteler-Dellacherie}} {\it{Theorem}} please see [27]. 

\begin{corollary}
Fractional Brownian motion is not a good integrator whenever $H \neq 1/2.$
\end{corollary}

\begin{proof}
This clearly follows from Theorem \ref{BD} by applying Theorem \ref{fbmnotsemi}.
\end{proof}

Thus, $B^H(t)$ is not a good integrator.  Either way, there have been various [successful] approaches in attempting to define an integral of this type, such as Wick, Stratonovich and pathwise, among others. In this section, we will introduce the pathwise approach and refer to [1] and [14] for further reading on pathwise and other types of stochastic integrals. 

	As shown previous, $B^H(t)$ is {\Holder} continuous of order $\beta \in [0, H),$ so it seems natural one might think to define a stochastic integral like that in \eqref{fBmsi} by using the traditional Riemann sum approach:

\begin{align}
\label{fBmsum}
\displaystyle \sum \limits_{i = 0}^{n-1} f(t_i) [B^H(t_{i+1}) - B^H(t_i)].
\end{align}
The problem with this approach though, is that $B^H(t)$, similar to $B(t),$ is not only non-differentiable, but it has almost surely sample paths of unbounded variation. Thus we are concerned that as we take $n \rightarrow \infty$, \eqref{fBmsum} will diverge (in probability). Even still, as we define a pathwise integral in this way, we take full advantage of the continuous sample paths of fBm, hence notice the similarity in the traditional definition of the derivative and the stochastic integrals defined below.  Since fBm has almost surely sample paths of unbounded variation, we must steer away from a Riemann sum approach and use the more generalized Riemann-Stieltjes definition.

\subsection{Symmetric, Forward and Backward Integrals}\label{sfb-fbm}

\begin{mydef} \label{symmetricfbm}
Let $H \in (0,1)$ and $\{f(t) ; t \in [0,T]\}$ be a process with integrable trajectories. The symmetric integral of $f(t)$ w.r.t. $B^H(t)$ is defined as 

\begin{align} \label{fbmsymmetric}
\int^{T}_{0} f(s) d^{\circ}B^H(s)  = \displaystyle \lim_{\epsilon \rightarrow 0} \frac{1}{2 \epsilon } \int^{T}_{0} f(s) \Big[ B^H(s+ \epsilon ) - B^H(s- \epsilon) \Big] ds,
\end{align}

provided that this limit exists in probability.
\end{mydef}

\begin{mydef} \label{forwardfbm}
Let $H \in (0,1)$ and $\{f(t) ; t \in [0,T]\}$ be a process with integrable trajectories. The forward integral of $f(t)$ w.r.t. $B^H(t)$ is defined as 

\begin{align}
\label{fBmforward}
\int^{T}_{0} f(s) d^{-}B^H(s)  = \displaystyle \lim_{\epsilon \rightarrow 0} \frac{1}{\epsilon } \int^{T}_{0} f(s) \frac{B^H(s+ \epsilon ) - B^H(s)}{\epsilon} ds,
\end{align}

provided that this limit exists in probability. Similarly, the backward integral is defined as
\begin{align}
\int^{T}_{0} f(s) d^{+}B^H(s)  = \displaystyle \lim_{\epsilon \rightarrow 0} \frac{1}{\epsilon } \int^{T}_{0} f(s) \frac{B^H(s- \epsilon ) - B^H(s)}{\epsilon} ds,
\end{align}
provided that the limit exists.
\end{mydef}

As we have defined the stochastic integral in these ways, we are still concerned as to when the convergence of these limits, if any, hold (in probability). And if so, how these definitions are related.  Since we are taking advantage of the {\Holder} continuity of the sample paths of $B^H(t),$ it seems reasonable that the $p$-variation behavior of these sample paths determines the existence of these limits. Indeed, this turns out to be true. In fact, if $H < 1/2,$ the limit in \eqref{fBmforward} does not even exist. 


\begin{theorem}\label{RSexist}
$[26, Theorem$ $6.2]$ Let $B^H(t)$ be a fractional Brownian motion with Hurst index $H.$ Let $u(t) : [0,T] \times \Omega \rightarrow \mathbb{R}$ be a stochastic process having bounded $p$-variation sample paths almost surely such that $p < 1/(1-H).$ Then the integral 
$$\int^T_0 u(s) dB^H(s)$$ exists almost surely in the Riemann-Stieltjes sense.
\end{theorem}

That is, with probability one, the integral given in Theorem \ref{RSexist} exists in the sense as defined above.  Moreover, it is important to note that in \eqref{fbmsymmetric}, if $f$ is a deterministic function with bounded variation, then the almost sure limit of Riemann sums gives us $\int^T_0 f(t) d^{\circ}B^H(t)$. Henceforth, since the sample paths of $B^H(t)$ are continuous, we can make use of the integration by parts formula.  Thus
\begin{align} \label{fbm-parts}
\int^T_0 f(t) d^{\circ}B^H(t)
& = f(T) B^H(T) - \int^T_0 B^H(s) df(s).
\end{align}
The convergence of the Riemann-Stieltjes integral on the right side of \eqref{fbm-parts} is the same as the convergence of the Riemann sums in the case when $f$ is deterministic in the symmetric integral.
 If we have two such stochastic processes as in Definition's \ref{symmetricfbm} and \ref{forwardfbm} , then the forward integral is related to the symmetric integral as defined earlier.

\begin{mydef}
Let $X(t)$ and $Y(t)$ be two continuous (respectively, locally bounded) processes. Their covariation $[X,Y]_t$ is defined to be 
\begin{align}
\displaystyle \lim_{\epsilon \rightarrow 0} \frac{1}{\epsilon} \int^t_0 \big(X(u + \epsilon) - X(u)\big)\big(Y(u + \epsilon) - Y(u)\big) du,
\end{align}
provided the limit exists in uniform convergence in probability.
\end{mydef}

\begin{proposition}
 $[1]$ Let $X(t)$ and $Y(t)$ be two continuous (respectively, locally bounded) processes. Then
\begin{align*}
\int^t_0 Y(u) d^{\circ}x(u) = \int^t_0 Y(u) d^{-}x(u) + [X, Y]_t
\end{align*}
holds provided that at least two of the three terms exist.
\end{proposition}

\subsection{On The Link Between Fractional and Stochastic Calculus} \label{frac-calc-link}

	Again, in this paper, we are interested in the relationship between fractional calculus and fractional Brownian motion. Indeed, this pathwise approach in defining the stochastic integral exhibits a strong link between the two.  To illustrate this link, we follow Z\"{a}hle in [30] and refer to this text for further details.  Consider two deterministic functions $h,g$ on $[0,T]$ which satisfy the conditions of Definition \ref{fractegral} and the fractal integral $\int^b_a h dg,$ $0 \leq a < b \leq T$ as defined in either case. Then, the following is valid:

$$\int^b_a h dg = \displaystyle \lim_{\epsilon \rightarrow 0} \int^b_a (I^{\epsilon}_{a+} h) dg.$$

If the degrees of differentiability of $h$ and $g$ sum to at least $1 - \epsilon,$ then by \eqref{fracibp} we obtain  ([1]):

$$\int^b_a (I^{\epsilon}_{a+} h) dg = \frac{1}{\Gamma(\epsilon)} \int^T_0 u^{\epsilon - 1} \int^b_a h(s) \frac{g_{b-}(s+u) - g_{b-}(s)}{u} ds du.$$
And hence
$$\int^b_a h dg = \displaystyle \lim_{\epsilon \rightarrow 0} \frac{1}{\Gamma(\epsilon)} \int^T_0 u^{\epsilon - 1} \int^b_a h(s) \frac{g_{b-}(s+u) - g_{b-}(s)}{u} ds du.$$

So, if we consider the stochastic process $f(t, \omega )$ and the fractional Brownian motion $B^H(t)$ in place of the deterministic functions $h$ and $g,$ we obtain the following definition.
\begin{mydef}
Let $\{f(t) ; t \in [0,T]\}$ be a stochastic process. The extended forward integral of $f(t)$ w.r.t. $B^H(t)$ is defined as

\begin{align}
\label{exforward}
\int^t_0 f(s) d^-B(s)
& = \displaystyle \lim_{\epsilon \rightarrow 0} \frac{1}{\Gamma(\epsilon)} \int^T_0 u^{\epsilon -1} \int^t_0 f(s) \frac{B^H(s+u) - B^H(s)}{s} ds du,
\end{align}

provided the limit exists in uniform convergence in probability.
\end{mydef}

As we derived \eqref{exforward} in the deterministic case, recall that we restricted the degrees of differentiability of $h$ and $g$ such that their sum is at least $1 - \epsilon.$ In the stochastic case, this is the same as restricting $f(t, \omega )$ and $B^H(t)$ to be {\Holder} continuous of order $\alpha$ and $\beta,$ respectively, such that $\alpha + \beta > 1.$ In this case, we get the existence of the Riemann-Stieltjes integral as defined above.

\subsection{{\Ito} Formula With Respect To Fractional Brownian Motion} \label{itoformfbm}

Due to this construction of the stochastic integral by taking advantage of the almost surely continuous sample paths of fBm, an analogue of {\Ito}'s formula exists by using the classical change of variables formula.

\begin{mydef}
Let $f(t)$ be a forward integrable process and let $\alpha(s)$ be a measurable process such that $\int^t_0 |\alpha(s)| ds < \infty$ almost surely for $t \geq 0.$ Then, for $t \geq 0,$the fractional forward process is 
\begin{align}
X(t) = x + \int^t_0 \alpha(s) ds + \int^t_0 f(s) d^- B^H(s).
\end{align}
We can also write $(7.8)$ as
$$d^- X(t) = \alpha(t) dt + f(t) d^- B^H(t),    X(0) = x.$$
\end{mydef}

\begin{theorem}
Let 
$$d^- X(t) = \alpha(t) dt + f(t) d^- B^H(t),    X(0) = x$$
Be a fractional forward process. Suppose $g(t,x) \in C^2$ and let $Y(t) = g(t, X(t)).$ Then, if $1/2 < H < 1,$ we get 
$$d^- Y(t) = \frac{\partial}{\partial t}g(t, X(t))dt +  \frac{\partial}{\partial x}g(t, X(t))d^-X(t).$$
\end{theorem}

\begin{proof}
Let $\{ t_0, t_1, ... , t_{n-1}, t_n\}$ be a partition of $[0,t]$ such that $t_0 = 0$ and $t_n = t.$ Then by using a Taylor expansion of $g(t,x)$ we get\\

	$Y(t) - Y(0)$
	\begin{align*}
	& = \displaystyle \sum \limits_{j = 0}^{n-1} Y(t_{j+1}) - Y(t_j)\\
	& = \displaystyle \sum \limits_{j = 0}^{n-1} g(t_{j+1},X(t_{j+1})) - g(,t_j,X(t_j))\\
	& = \displaystyle \sum \limits_{j = 0}^{n-1} \frac{\partial}{\partial t}g(t_{j},X(t_{j})) \bigtriangleup 	t_j + \displaystyle \sum \limits_{j = 0}^{n-1}\frac{\partial}{\partial t}g(,t_j,X(t_j)) \bigtriangleup X(t_j)\\
	& + \frac{1}{2} \displaystyle \sum \limits_{j = 0}^{n-1}\frac{\partial^2}{\partial t^2}g(,t_j,X(t_j)) (\bigtriangleup X(t_j))^2 + \displaystyle \sum \limits_{j = 0}^{n-1} O((\bigtriangleup t_j)^2) + O((\bigtriangleup X(t_j))^2),
\end{align*}
where $\bigtriangleup X(t_j) =  X(t_{j+1}) - X(t_j)$ and $\bigtriangleup t_j = t_{j+1} - t_j.$ But, $1/2 < H < 1,$ hence the quadratic variation of $X(t)$ is zero. Thus, as $M_n$ tends to zero, we get
$$Y(t) - Y(0) = \int^t_0 \frac{\partial}{\partial s}g(s,X(s)) ds + \int^t_0 \frac{\partial}{\partial x}g(s,X(s)) d^-X(s).$$
\end{proof}

Unlike the case of the {\Ito} stochastic integral with respect to the Brownian motion, the pathwise attempt to define the integral with respect to fractional Brownian motion does not have zero mean.  Moreover, there is no easy formula for its variance, where in order to compute the mean and variance of this integral, we need the techniques of the Malliavin calculus, which we will not do here. Furthermore, when $H < 1/2,$ more care is needed in the construction of the pathwise definition of \eqref{fBmsi} as more difficulties arise. One example being, as mentioned earlier, that if $H < 1/2,$ the quadratic variation of $B^H(t)$ is infinite, and hence the limit in \eqref{fBmforward} does not even exist. For further details we refer to [1] and [14].

	We have shown that the well known Brownian motion is just a specific case of the fractional Brownian motion, when $H = 1/2$, thus splitting this generalization into three distinct families. We first described Brownian motion and its properties in order to develop a stochastic calculus with respect to $B(t)$. In this development, {\Ito} took advantage of certain properties, such as independent increments and $B(t)$ being a martingale. Because of this, it turns out that,  by the $Bichteler-Dellacherie$ $Theorem$, Brownian motion is a good integrator. Whereas this is not the case for fBm with values of $H$ not equal to $1/2$, since it is not a semimartingale. On this note, it is interesting to remark that Brownian motion is a better integrator (in the sense that $B^H(t)$ is a good integrator if $H = 1/2$ and is not a good integrator otherwise) than fBm even though sample paths of fBm, for values of $H \in (1/2 , 1)$, are smoother than sample paths of $B^{\frac{1}{2}}(t).$  Even still, in the midst of this, there have been many successful approaches in characterizing a stochastic integral with respect to fBm. In this paper we chose to discuss the pathwise approach for its interesting revelations on the link between fractional calculus and fractional Brownian motion.  We refer to Nualart in [23] and Z\"{a}hle in [30] for further details on this relationship.

\end{doublespace}
\pagebreak


\end{document}